\documentclass{amsart}
\usepackage{amsfonts}
\usepackage{latexsym,amssymb,graphics}

\newtheorem{theorem}{Theorem}

\newtheorem{lemma}{Lemma}

\newtheorem{remark}{Remark}

\def\tbigcap{\bigcap}

\begin{document}
\title[Kronecker constants]{Upper and Lower Bounds for Kronecker constants
of Three-Element Sets of Integers}
\author{Kathryn E. Hare}
\address{Dept. of Pure Mathematics\\
University of Waterloo\\
Waterloo, Ont.\\
Canada, N2L 3G1}
\email{kehare@uwaterloo.ca}
\thanks{This research was supported in part by NSERC. The first author would
like to thank the Dept. of Mathematics at University of Hawaii for their
hospitality while this research was being done.}
\author{L. Thomas Ramsey}
\address{Dept. of Mathematics\\
University of Hawaii at Manoa\\
Honolulu, Hi\\
USA, 96822}
\email{ramsey@math.hawaii.edu}

\begin{abstract}
Various upper and lower bounds are provided for the (angular) Kronecker constants of sets of integers.  Some examples are provided where the bounds are attained.  It is proved that $5/16$ bounds the angular Kronecker constants of 3-element sets of positive integers. However, numerous examples suggest that the minimum upper bound is $1/4$ for 3-element sets of positive integers.
\end{abstract}

\maketitle

\section{Introduction}

A subset $S$ of the dual of a compact, abelian group $G$ is called an $%
\varepsilon $-Kronecker set if for every continuous function $f$ mapping $S%
\mathbb{\ }$into $\mathbb{T}$, the set of complex numbers of modulo $1,$
there exists $x\in G$ such that
$$
\left\vert \gamma (x)-f(\gamma )\right\vert <\varepsilon \text{ for all }%
\gamma \in S\text{.}
$$
The infimum of such $\varepsilon $ is called the Kronecker constant, $\kappa
(S)$.

We continue with the notation of \cite{HR}, and define an angular Kronecker constant
$\alpha(S) \in [0,1/2]$ such that
$$
\kappa(S)=\left|e^{2\pi i \alpha(S)}-1 \right|
$$
In this note, the terminology, notations and results of \cite{HR} will be used.

This note is a supplement to \cite{HR}, giving some bounds for angular Kronecker constants of three-element sets of integers:
\begin{itemize}
\item
Theorem \ref{upperlowerbound} uses the full machinery of \cite{HR} to produce both upper and lower bounds for 3-element sets of non-zero, relatively prime integers with distinct absolute values.
\item
Theorem \ref{upperlowerbound} and two results of \cite{HR} are used to prove that $\alpha(S) \le 5/16$ for all $S$ consisting of non-zero integers with distinct absolute values.  See Theorem \ref{T:fivesixteenths}.  Examples suggest that one should be able to lower this upper bound to $1/4$ for $5/16$.
\end{itemize}

\section{Bounding Estimates for (Most) Three Element Sets of Integers}

\cite{HR} provides an easy upper bound (better than $1/2$) for any finite set $S$ of integers that
does not contain $0$:  
$$
\alpha(S) \le \dfrac 1 2 -\dfrac 1 {2d}
$$
where $d$ is the size of $S$.

When $d=2$, this gives $\alpha (S)\leq $ $1/4$ and this trivial upper bound
is sharp if and only if $S=\{-n,\,n\}$. For $d=3$, the bound is $1/3$ and
\cite{HR} shows that this is sharp when $S=\{-n,\,n,\,2n\}.$ A
consequence of our next theorem is that the angular Kronecker constant is
strictly less than $1/3$ for all other three element sets (that exclude $0$).

\begin{theorem}
\label{upperlowerbound}Suppose $\left\vert n_{1}\right\vert ,\left\vert
n_{2}\right\vert ,\left\vert n_{3}\right\vert $ are distinct and $\gcd
(n_{1},n_{2},n_{3})=1$. Assume $m=$ $\gcd (n_{2},n_{3})$ and that $(1/m,0)$
and $(r/n_{3},m/n_{3})$ $\ $generate the lattice $\mathcal{K}$ for some $r>0$%
. We have the following bounds on angular Kronecker constants:
$$
\alpha \{n_{1},n_{2},n_{3}\}\geq \frac{\left\vert n_{3}\right\vert }{%
2(r(\left\vert n_{2}\right\vert +\left\vert n_{3}\right\vert )+m(\left\vert
n_{1}\right\vert +\left\vert n_{3}\right\vert ))}
$$
and%
$$
\alpha \{n_{1},n_{2},n_{3}\}\leq \frac{E_{1}\left( 2\left\vert
n_{1}\right\vert \left\vert n_{2}\right\vert +\left\vert n_{3}\right\vert
(\left\vert n_{1}\right\vert +\left\vert n_{2}\right\vert )\right) }{%
\left\vert n_{3}\right\vert (\left\vert n_{1}\right\vert +\left\vert
n_{2}\right\vert )}
$$
where
$$
E_{1}=\max \left( \frac{m}{2(\left\vert n_{2}\right\vert +\left\vert
n_{3}\right\vert )},\frac{r}{2(\left\vert n_{1}\right\vert +\left\vert
n_{3}\right\vert )},\frac{\left\vert n_{3}\right\vert +2rm}{2\left(
r(\left\vert n_{2}\right\vert +\left\vert n_{3}\right\vert )+m(\left\vert
n_{1}\right\vert +\left\vert n_{3}\right\vert )\right) }\right) .
$$
\end{theorem}

\begin{proof}
Temporarily fix $(x,y)$ and put $\beta =x-ry/m$. For an integer $t$, put $%
\Delta _{t}=y-tm/n_{3}$. With this notation, the angular Kronecker constant
is the least constant $E$ such that for each $\beta $ there are integers $%
s,t $ such that
$$
\begin{aligned}
(2,3) &\quad:\quad\frac{\left\vert n_{3}\right\vert }{\left\vert n_{2}\right\vert
+\left\vert n_{3}\right\vert }\left\vert \Delta _{t}\right\vert \leq E \\
(1,3) &\quad:\quad\frac{\left\vert n_{3}\right\vert }{\left\vert n_{1}\right\vert
+\left\vert n_{3}\right\vert }\left\vert \beta -\frac{s}{m}+\Delta _{t}\frac{%
r}{m}\right\vert \leq E \\
(1,2) &\quad:\quad\frac{\left\vert n_{2}(\beta -\frac{s}{m}+\Delta _{t}\frac{r}{m}%
)-n_{1}\Delta _{t}\right\vert }{\left\vert n_{1}\right\vert +\left\vert
n_{2}\right\vert }\leq E.
\end{aligned}
$$

The requirements $(1,3)$ and $(2,3)$ can be satisfied if and only for each $%
\beta $ there are integers $s,t$ such that $\Delta _{t}\in
J_{1}(s,E)\tbigcap J_{2}(s,E)$ where $J_{1},J_{2}$ are the intervals
$$
\begin{aligned}
J_{1}(s,E) &\quad=\quad\left[ -E\left( \frac{\left\vert n_{2}\right\vert +\left\vert
n_{3}\right\vert }{\left\vert n_{3}\right\vert }\right) ,E\left( \frac{%
\left\vert n_{2}\right\vert +\left\vert n_{3}\right\vert }{\left\vert
n_{3}\right\vert }\right) \right] \text{ and} \\
J_{2}(s,E) &\quad=\quad\left[ \frac{m}{r}\left( -E\left( \frac{\left\vert
n_{1}\right\vert +\left\vert n_{3}\right\vert }{\left\vert n_{3}\right\vert }%
\right) -\beta +\frac{s}{m}\right) ,\frac{m}{r}\left( E\left( \frac{%
\left\vert n_{1}\right\vert +\left\vert n_{3}\right\vert }{\left\vert
n_{3}\right\vert }\right) -\beta +\frac{s}{m}\right) \right] .
\end{aligned}
$$
If there is some choice of $\beta $ such that $J_{1}(s,E)\tbigcap J_{2}(s,E)$
is empty for all $s\in \mathbb{Z}$, then clearly $(1,3)$ and $(2,3)$ cannot
be simultaneously satisfied for that choice of $E$. Hence it is necessary
that for each $\beta $ there be an integer $s$ with the right end of $%
J_{1}(s,E)\geq $ left end of $J_{2}(s,E)$ and the right end of $%
J_{2}(s,E)\geq $ left end of $J_{1}(s,E)$. This implies that if we let
$$
c(E)=\frac{E}{\left\vert n_{3}\right\vert }\left( \frac{r}{m}(\left\vert
n_{2}\right\vert +\left\vert n_{3}\right\vert )+\left\vert n_{1}\right\vert
+\left\vert n_{3}\right\vert \right) ,
$$
then a necessary condition on $E$ is that for each $\beta $ there is an
integer $s$ satisfying
$$
-c(E)\leq -\beta +\frac{s}{m}\leq c(E).
$$
However, if $\beta =1/2m$ and $c(E)<1/2m$ then this inequality cannot hold
for any integer $s$. Thus a necessary condition for $E$ (a lower bound on
the Kronecker constant) is that $E$ must satisfy%
$$
\frac{E}{\left\vert n_{3}\right\vert }\left( \frac{r}{m}(\left\vert
n_{2}\right\vert +\left\vert n_{3}\right\vert )+\left\vert n_{1}\right\vert
+\left\vert n_{3}\right\vert \right) \geq \frac{1}{2m}
$$
and this gives the lower bound stated in the theorem.

Now put $E_{1}=E_{0}+\lambda $, where $E_{0}$ is the lower bound and $%
\lambda >0$ is to be determined. Then $c(E_{1})\geq 1/2m$, so for any $x,y$
there will be an integer $s$ such that $J_{1}(s,E_{1})\tbigcap
J_{2}(s,E_{1}) $ is non-empty. If, in addition, the length of the overlap of
the two intervals is at least $m/\left\vert n_{3}\right\vert $, then we can
be sure that for any choice of $y$, there will be an integer $t$ with $%
\Delta _{t}=y-tm/n_{3}\in J_{1}(s,E_{1})\tbigcap J_{2}(s,E_{1})$.

The length of the overlap will be either%
\begin{eqnarray}
2E_{1}\left( \frac{\left\vert n_{2}\right\vert +\left\vert n_{3}\right\vert
}{\left\vert n_{3}\right\vert }\right) \text{ if }J_{1} &\subseteq &J_{2}
\label{1} \\
2E_{1}\left( \frac{\left\vert n_{1}\right\vert +\left\vert n_{3}\right\vert
}{\left\vert n_{3}\right\vert }\right) \frac{m}{r}\text{ if }J_{2}
&\subseteq &J_{1}  \label{2}
\end{eqnarray}%
or%
\begin{equation}
E_{1}\left( \frac{\left\vert n_{2}\right\vert +\left\vert n_{3}\right\vert }{%
\left\vert n_{3}\right\vert }\right) +\left( E_{1}\frac{\left\vert
n_{1}\right\vert +\left\vert n_{3}\right\vert }{\left\vert n_{3}\right\vert }%
-\left\vert -\beta +\frac{s}{m}\right\vert \right) \frac{m}{r}\text{
otherwise.}  \label{3}
\end{equation}

Obviously, if
$$
E_{1}\geq \max \left( \frac{m}{2(\left\vert n_{2}\right\vert +\left\vert
n_{3}\right\vert )},\frac{r}{2(\left\vert n_{1}\right\vert +\left\vert
n_{3}\right\vert )}\right)
$$
then both (\ref{1}) and (\ref{2}) will be at least $m/\left\vert
n_{3}\right\vert $.

Now consider (\ref{3}). Note that the choice of $E_{0}$ ensures that for any
$\beta $ there is an integer $s$ with
$$
-\left\vert -\beta +\frac{s}{m}\right\vert \geq -\frac{1}{2m}=\frac{-E_{0}}{%
\left\vert n_{3}\right\vert }\left( \frac{r}{m}(\left\vert n_{2}\right\vert
+\left\vert n_{3}\right\vert )+\left\vert n_{1}\right\vert +\left\vert
n_{3}\right\vert \right) .
$$
Thus
$$
E_{1}\left( \frac{\left\vert n_{2}\right\vert +\left\vert n_{3}\right\vert }{%
\left\vert n_{3}\right\vert }\right) +\left( E_{1}\frac{\left\vert
n_{1}\right\vert +\left\vert n_{3}\right\vert }{\left\vert n_{3}\right\vert }%
-\left\vert -\beta +\frac{s}{m}\right\vert \right) \frac{m}{r}\geq \frac{%
\lambda \left( r(\left\vert n_{2}\right\vert +\left\vert n_{3}\right\vert
)+m(n_{1}+\left\vert n_{3}\right\vert )\right) }{r\left\vert
n_{3}\right\vert }.
$$
and therefore we can satisfy the inequality (\ref{3}) $\geq m/\left\vert
n_{3}\right\vert $ if we choose
$$
\lambda \geq \frac{rm}{r(\left\vert n_{2}\right\vert +\left\vert
n_{3}\right\vert )+m(\left\vert n_{1}\right\vert +\left\vert
n_{3}\right\vert )},
$$
that is,
$$
E_{1}\geq \frac{\left\vert n_{3}\right\vert +2rm}{2\left( r(\left\vert
n_{2}\right\vert +\left\vert n_{3}\right\vert )+m(\left\vert
n_{1}\right\vert +\left\vert n_{3}\right\vert )\right) }.
$$
The construction ensures that if we take $E=E_{1}$, then both inequalities $%
(1,3)$ and $(2,3)$ can be simultaneously satisfied for each $x,y,$ with a
suitable choice of integers $s,t$.

But then,
$$
\left\vert \Delta _{t}\right\vert \leq E_{1}\left( \frac{\left\vert
n_{2}\right\vert +\left\vert n_{3}\right\vert }{\left\vert n_{3}\right\vert }%
\right) \text{ and }\left\vert \beta -\frac{s}{m}+\Delta _{t}\frac{r}{m}%
\right\vert \leq E_{1}\left( \frac{\left\vert n_{1}\right\vert +\left\vert
n_{3}\right\vert }{\left\vert n_{3}\right\vert }\right)
$$
and therefore
$$
\begin{aligned}
(1,2) &\quad \leq \quad\frac{\left\vert n_{2}\right\vert E_{1}\left( \frac{\left\vert
n_{1}\right\vert +\left\vert n_{3}\right\vert }{\left\vert n_{3}\right\vert }%
\right) +\left\vert n_{1}\right\vert E_{1}\left( \frac{\left\vert
n_{2}\right\vert +\left\vert n_{3}\right\vert }{\left\vert n_{3}\right\vert }%
\right) }{\left\vert n_{1}\right\vert +\left\vert n_{2}\right\vert } \\
&\quad=\quad\frac{E_{1}\left( 2\left\vert n_{1}\right\vert \left\vert
n_{2}\right\vert +\left\vert n_{3}\right\vert (\left\vert n_{1}\right\vert
+\left\vert n_{2}\right\vert )\right) }{\left\vert n_{3}\right\vert
(\left\vert n_{1}\right\vert +\left\vert n_{2}\right\vert )}.
\end{aligned}
$$
This verifies the claimed upper bound on the angular Kronecker constant.
\end{proof}

To apply this result to show that the Kronecker constant is less than $1/3$
for any three element set not containing $0$, other than the sets $%
\{-n,n,2n\},$ it is convenient to first record some preliminary
calculations. The arguments are elementary and we will only give the main
idea for each.

\begin{lemma}
\label{L1}Assume $0<\left\vert n_{1}\right\vert <n_{2}<n_{3}$ and that
integers $r,m\geq 1$.

(i) If $E_{1}\leq n_{3}/4(\left\vert n_{1}\right\vert +n_{3})$, then
$$
\frac{E_{1}\left( 2\left\vert n_{1}\right\vert n_{2}+n_{3}(\left\vert
n_{1}\right\vert +n_{2})\right) }{n_{3}(\left\vert n_{1}\right\vert +n_{2})}%
\leq \frac{5}{16}.
$$

(ii) If $r+m\geq 5$, then%
$$
\frac{2\left\vert n_{1}\right\vert n_{2}+n_{3}(\left\vert n_{1}\right\vert
+n_{2})}{(\left\vert n_{1}\right\vert +n_{2})\left(
r(n_{2}+n_{3})+m(\left\vert n_{1}\right\vert +n_{3})\right) }\leq \frac{5}{16}%
.
$$

(iii) If $(r,m)\neq (1,1)$, then
$$
\frac{\left( 2\left\vert n_{1}\right\vert n_{2}+n_{3}(\left\vert
n_{1}\right\vert +n_{2})\right) (2+\min (r,m))}{4(\left\vert
n_{1}\right\vert +n_{2})\left( r(n_{2}+n_{3})+m(\left\vert n_{1}\right\vert
+n_{3})\right) }\leq \frac{5}{16}.
$$
\end{lemma}

\begin{proof}
(i) Consider the cases $\left\vert n_{1}\right\vert \geq n_{2}/2$ and $%
\left\vert n_{1}\right\vert \leq n_{2}/2$ separately. In the first case use
the fact that $4n_{1}^{2}\geq 2\left\vert n_{1}\right\vert n_{2}$; in the
second, use the inequality $n_{3}(\left\vert n_{1}\right\vert +n_{2})\geq
3\left\vert n_{1}\right\vert n_{2}$.

(ii) The key idea here is the inequality $(\left\vert n_{1}\right\vert
+n_{2})\left( rn_{2}+m\left\vert n_{1}\right\vert \right) \geq 7\left\vert
n_{1}\right\vert n_{2}$.

(iii) Here it is convenient to put $s=\min (r,m)$ and write $r+m=2s+l$ for $%
l\geq 0$. \ Then consider the two possibilities: $s\geq 2$, $l\geq 0$ and $%
s=1$, $l\geq 1$.
\end{proof}

\begin{lemma}
\label{L4}Assume $0<\left\vert n_{1}\right\vert <n_{2}<n_{3}$. If $r=m=1$,
then $n_{1}<0$ and $n_{3}=\left\vert n_{1}\right\vert +n_{2}.$
\end{lemma}

\begin{proof}
These assumptions imply that there is an integer $t$ with $tn_{2}\equiv
tn_{1}\equiv 1\text{\thinspace mod\thinspace }n_{3}$. It follows that $\gcd
(t,n_{3})=1$ and $t(n_{1}-n_{2})\equiv 0\text{\thinspace mod\thinspace }%
n_{3}.$ Hence $n_{3}$ divides $(n_{2}-n_{1})$. This is not possible if $%
n_{1}>0$ and can only occur with $n_{1}<0$ if $n_{2}-n_{1}=n_{3}$, i.e., $%
n_{3}=n_{2}+\left\vert n_{1}\right\vert $ since $\left\vert n_{1}\right\vert
+n_{2}<2n_{3}$.
\end{proof}

\begin{theorem}\label{T:fivesixteenths}
Let $S=\{n_{1},n_{2},n_{3}\}$ where $n_{j}\neq 0$. If $\{n_{1},n_{2},n_{3}\}%
\neq \{-n,n,2n\}$ for some integer $n$, then $\alpha (S)\leq 5/16.$
\end{theorem}

\begin{proof}
When the $\left\vert n_{j}\right\vert $'s are not distinct, \cite{HR} proves that the angular Kronecker constant is at most $3/10$. Thus there is
no loss of generality in assuming $0<\left\vert n_{1}\right\vert
<n_{2}<n_{3} $.

Assume $m=\gcd (n_{2},n_{3})$ and that the lattice is generated by $(1/m,0)$%
, $(r/n_{2},m/n_{3})$ where we choose $-n_{3}/2m<r\leq n_{3}/2m$. If $r=0$
we are in the rectangular lattice case developed in \cite{HR}, for which
the angular Kronecker constant is at most $1/5$.

We can assume $r>0$ by replacing $n_{1}$ by $-n_{1},$ if necessary. Of
course, $m\leq n_{3}/2$.

According to the previous theorem it will be enough to prove that
$$
\frac{E_{1}\left( 2\left\vert n_{1}\right\vert n_{2}+n_{3}(\left\vert
n_{1}\right\vert +n_{2})\right) }{n_{3}(\left\vert n_{1}\right\vert +n_{2})}%
\leq \frac{5}{16}
$$
where
$$
E_{1}=\max \left( \frac{m}{2(n_{2}+n_{3})},\frac{r}{2(\left\vert
n_{1}\right\vert +n_{3})},\frac{n_{3}+2rm}{2\left(
r(n_{2}+n_{3})+m(\left\vert n_{1}\right\vert +n_{3})\right) }\right) .
$$

As $r,m\leq n_{3}/2$, if $E_{1}=m/(2(n_{2}+n_{3}))$ or $r/(2(\left\vert
n_{1}\right\vert +n_{3}))$, then $E_{1}\leq n_{3}/(4(\left\vert
n_{1}\right\vert +n_{3}))$ and calling upon Lemma \ref{L1}(i) gives the
desired result.

Otherwise,
$$
E_{1}=\frac{n_{3}+2rm}{2\left( r(n_{2}+n_{3})+m(\left\vert n_{1}\right\vert
+n_{3})\right) }.
$$
Since $r\leq n_{3}/2m$,
$$
E_{1}\leq \frac{n_{3}}{r(n_{2}+n_{3})+m(\left\vert n_{1}\right\vert +n_{3})}
$$
and thus Lemma \ref{L1}(ii) gives the bound if $r+m\geq 5$.

So suppose $r+m\leq 4$ and hence $\max (r,m)\leq 3.$ If $n_{3}\geq 12$, $%
\max (r,m)\leq n_{3}/4,$ thus
$$
E_{1}\leq \frac{n_{3}+n_{3}\min (r,m)/2}{2(r(n_{2}+n_{3})+m(\left\vert
n_{1}\right\vert +n_{3}))}=\frac{n_{3}(2+\min (r,m))}{4(r(n_{2}+n_{3})+m(%
\left\vert n_{1}\right\vert +n_{3}))}.
$$
and if $(r,m)\neq (1,1)$ we can appeal to Lemma \ref{L1}(iii)

From Lemma \ref{L4} we know that the case $r=m=1$ can only arise if $%
\{\left\vert n_{1}\right\vert ,n_{2},n_{3}\}$ is a sum set. Then $E_{1}\leq
(n_{3}+2)/6n_{3}$. Moreover, $\left\vert n_{1}\right\vert n_{2}\leq
n_{3}^{2}/4$, so for $n_{3}\geq 12$,%
$$
\frac{E_{1}\left( 2\left\vert n_{1}\right\vert n_{2}+n_{3}(\left\vert
n_{1}\right\vert +n_{2})\right) }{n_{3}(\left\vert n_{1}\right\vert +n_{2})}%
\leq \left( \frac{1}{6}+\frac{1}{3n_{3}}\right) \frac{3}{2}\leq \frac{7}{24}<%
\frac{5}{16}.
$$

For the three element sets with $n_{3}<12$, we apply our computer algorithm.
The greatest angular Kronecker constant is $1/4,$ occuring with the set $%
\{1,2,3\}$ (and its multiples).
\end{proof}

\begin{remark}
We conjecture that $\alpha \{n_{1},n_{2},n_{3}\}\leq 1/4$ for all three
element sets other than $\{-n,n,2n\}$.  In \cite{HR} this was proved for the rectangular lattice case and for sum sets. As well, we have run our computer
algorithm on all three element sets of positive integers with $n_{3}\leq 50$
and the greatest Kronecker constant is $1/4$, occuring only on the integer
multiples of $\{1,2,3\}$. 
\end{remark}

\end{document}